\documentclass{amsart}

\newtheorem{theorem}{Theorem}[section]
\newtheorem{lemma}[theorem]{Lemma}

\theoremstyle{definition}
\newtheorem{definition}[theorem]{Definition}

\theoremstyle{remark}
\newtheorem{remark}[theorem]{Remark}

\numberwithin{equation}{section}

\begin{document}

\title{Stickiness of KAM tori for higher dimensional beam equation}

\author{Xiucui Song*}
\address{*School of Mathematical Sciences, Fudan University. Rm 1818, Guanghua East Tower, 220 Handan Road, Shanghai 200433 China}
%\curraddr{Department of Mathematics and Statistics,
%Case Western Reserve University, Cleveland, Ohio 43403}
\email{xiucuisong12@fudan.edu.cn}
\thanks{The second author was supported in part by NNSFC Grant \#11101059.}

\author{Hongzi Cong**}
\address{**School of Mathematical Sciences, Dalian University of Technology, Dalian, Liaoning 116024, China}
\email{conghongzi@dlut.edu.cn(corresponding author)}

%    General info

%\date{January 1, 2001 and, in revised form, June 22, 2001.}



\begin{abstract}
This paper is concerned with the stickiness of invariant tori obtained
by KAM technics (so-called KAM tori) for higher dimensional beam equation. We prove that the KAM tori are sticky, i.e. the solutions starting in the $\delta$-neighborhood of KAM torus still stay close to the KAM torus for a polynomial long time such as $|t|\leq \delta^{-\mathcal{M}}$ with any $\mathcal{M}\geq 0$, by constructing a partial normal form of higher order, which satisfies $p$-tame property, around the KAM torus.
\end{abstract}

\maketitle

%\section*{This is an unnumbered first-level section head}
%This is an example of an unnumbered first-level heading.

%% The correct journal style for \specialsection is all uppercase; a known bug
%% in amsart.cls prevents this, so input must be uppercase until it is fixed.
%\specialsection*{This is a Special Section Head}
%\specialsection*{THIS IS A SPECIAL SECTION HEAD}
%This is an example of a special section head%
%%%%%%%%%%%%%%%%%%%%%%%%%%%%%%%%%%%%%%%%%%%%%%%%%%%%%%%%%%%%%%%%%%%%%%%%
%\footnote{Here is an example of a footnote. Notice that this footnote
%text is running on so that it can stand as an example of how a footnote
%with separate paragraphs should be written.
%\par
%And here is the beginning of the second paragraph.}%
%%%%%%%%%%%%%%%%%%%%%%%%%%%%%%%%%%%%%%%%%%%%%%%%%%%%%%%%%%%%%%%%%%%%%%%%

\section{Introduction and main results}
\subsection{Introduction and main results}
Since 1990's, KAM theory and Nekhoroshev theorem have a great development for infinite-dimensional Hamiltonian systems. See \cite{BB1},\cite{Bour1}-\cite{B3},\cite{Bour5},\cite{Bour6},\cite{C-W},\cite{E-K}-\cite{LY2},\cite{P1},\cite{W} and \cite{Bam1}-\cite{BN},\cite{BFG},\cite{B6},\cite{B2004},\cite{DS} for example.
KAM theory concerns the preservation and linear stability
of a majority of the non-resonant invariant tori (so-called KAM tori), and Nekhoroshev theorem concerns exponential lower bounds for the stability time (so-called effective stability).
Note that the trajectories lying in KAM tori clearly have an infinite stability time (so-called perpetual stability).
Therefore one can
also expect that, for a trajectory starting near a KAM torus, the stability
time is much larger than the one predicted by Nekhoroshev theorem (so-called stickiness). For finite dimensional Hamiltonian systems,
results concerning this `stickiness' of KAM tori have been obtained in  \cite{MG}-\cite{PW}. It is a natural question of the effective stability of the tori for infinite dimensional Hamiltonian systems. Recently, such a result about the long time stability for nonlinear Schr$\ddot{\mbox{o}}$dinger equation and nonlinear wave equation
has been given in \cite{CLY} and \cite{CGL}. The basic idea is that due to the suitable $p$-tame property, which generalized the key idea in \cite{BG}, and constructing a partial normal form of higher order, and then one can show that the solution, which starts in the $\delta$-neighbourhood of a KAM torus, still stays in the $\delta$-neighbourhood of the KAM torus in a polynomial long time.

In this paper, we consider $d$-dimensional ($d\geq 1$) beam equation
\begin{equation}\label{26}
u_{tt}+(-\triangle+M_{\xi})^2u+\varepsilon f(u)=0,\qquad x\in\mathbb{T}^d,
\end{equation}
where
$M_{\xi}$ is a real
Fourier multiplier defined by
\begin{equation*}\label{091913}
M_{\xi} \phi_{j}(x)=\xi_{j}\phi_{j}(x),
\end{equation*}
with $\xi=(\xi_{\textbf{j}})_{\textbf{j}\in\mathbb{Z}^d}\in\Pi\subset\mathbb{R}^{\mathbb{Z}^d}$ and
\begin{equation*}
\phi_{\textbf{j}}(x)=\frac{1}{(2\pi)^{d/2}}e^{\sqrt{-1}\langle \textbf{j},x \rangle},
\end{equation*}
and
$f(u)$ is a real-analytic function near $u = 0$ with $f (0) = f'(0) = 0.$ For most of $\xi\in\Pi$ and sufficiently small $\varepsilon$, the existence of KAM tori of equation (\ref{26}) was given in \cite{GY2003} ($d=1$ and $M_{\xi}$ is replaced by a fixed constant potential $m$), \cite{GY2006} ($d\geq 1$), \cite{GY20061} ($d\geq 1$, $M_{\xi}$ is replaced by a constant potential $m$ and $m$ is considered as a parameter) and \cite{XG} ($d\geq 1$, $M_{\xi}$ is replaced by a fixed constant potential $m$) respectively. However, there is nothing known about the long time stability about the KAM tori for equation (\ref{26}). In the present paper, we will prove
that "most" of KAM tori for equation (\ref{26}) are sticky.
More precisely, we have the following theorem:
\begin{theorem}\label{T4}Consider the higher dimensional beam equation
\begin{equation*}\label{26}
u_{tt}+(-\triangle+M_{\xi})^2u+\varepsilon f(u)=0,\qquad x\in\mathbb{T}^d.
\end{equation*}
There exists a large subset $\tilde\Pi\subset \Pi$, such that for each $\xi\in\tilde\Pi$ the KAM torus $\mathcal{T}_{\xi}$ of equation
(\ref{26}) is stable in long time. Precisely, for arbitrarily given $\mathcal M$ with
$0\leq \mathcal{M}\leq C(\varepsilon)$ (where $C(\varepsilon)$ is a constant depending on $\varepsilon$ and $C(\varepsilon)\rightarrow\infty$ as $\varepsilon\rightarrow0$) and $p\geq 8(\mathcal{M}+7)^{4}+1$, there is a
small positive $\delta_0$ depending on $n,p$ and
$\mathcal{M}$, such that for any
$0<\delta<\delta_0$ and any solution $u(t,x)$ of equation (\ref{26})
with the initial datum satisfying
$${d}_{H^p(\mathbb{T}^d)}(u(0,x),\mathcal{T}_{\xi})
:=\inf_{w\in\mathcal{T}_\xi}||u(0,x)-w||_{H^p(\mathbb{T}^d)}\leq \delta,$$
then
\begin{equation*}{d}_{H^p(\mathbb{T}^d)}(u(t,x),\mathcal{T}_{\xi})
:=\inf_{w\in\mathcal{T}_\xi}||u(t,x)-w||_{H^p(\mathbb{T}^d)}\leq
2\delta,\qquad \mbox{for all} \ |t|\leq
{\delta}^{-\mathcal{M}}.
\end{equation*}
\end{theorem}

\subsection{Further discussion}
As the paper \cite{BG} says, the key points to prove the long time stability result are: one is that to define a suitable $p$-tame property ($p$-tame norm) and to prove the $p$-tame property persistence under normal form iterative (some estimates about $p$-tame norm); the other is that some nonresonant conditions should be satisfied.

Following the idea in \cite{CGL} (or \cite{CLY}), it is easy to define the suitable $p$-tame norm and show the $p$-tame property persists under KAM iterative procedure and normal form iterative procedure. However it is not obvious that the nonresonant conditions hold true, since the eigenvalues of Laplacian operator are multiple with $d\geq1$ under periodic boundary conditions. We overcome this difficulty by the observation that there are some symmetry in the nonlinearity (see (\ref{04041}) and (\ref{04051})) and the regularity in the nonlinearity (see the definition of $p$-tame norm (\ref{091402}) where noting $\lfloor JW_z\rceil_{D(s,r)\times\Pi}$ is an operator form $\ell_{b,p}^2$ to $\ell_{b,p+2}^2$), which is actually used in \cite{GY2006} where a normal form of order $2$ is given. To obtain a partial normal form of high order, one has to face a more complicated small divisor problem. After a careful calculation, we prove that the nonresonant conditions are satisfied. Finally, we point out that the method in our paper can not be applied to deal with the problem of the long time stability of KAM tori for $d$-dimensional nonlinear Schr$\ddot{\mbox{o}}$dinger equation due to lack of the regularity in the nonlinear.
  
The paper is organized as follows. In section 2, we give some basic notations and definitions of $p$-tame norm for a Hamiltonian vector field. In section 3, we construct a norm form of order 2, which satisfies $p$-tame property, around the KAM tori based on the standard KAM method (see Theorem \ref{T1}), and a partial normal form of order $\mathcal{M}+2$ in the neighbourhood of the KAM tori (see Theorem \ref{thm7.1}). Since the iterative procedure is parallel to \cite{CLY}, we only prove the measure estimate in detail. Finally, due to the partial normal form of order $\mathcal{M}+2$ and $p$-tame property, we show the KAM tori are stable in a long time (see Theorem\ref{T3}). In section 4, we finish the proof of Theorem \ref{T4}. In section 5, we give the proof of the measure estimate. In section 6, we list some properties of $p$-tame norm. these properties are used in the proof of Theorem \ref{T1} and Theorem \ref{thm7.1} to ensure the $p$-tame property surviving under KAM iterative procedure and normal form iterative procedure.

\section{The definition of $p$-tame norm for a Hamiltonian vector field}
We will define $p$-tame norm for a Hamiltonian vector field as in \cite{CLY} in this section. First we introduce the
functional setting and the main notations concerning infinite dimensional Hamiltonian systems. Given $n\geq1$, let $S=\{\textbf{j}_1,\dots,\textbf{j}_n\}\subset\mathbb{Z}^d$ and $\mathbb{Z}_1^d:=\mathbb{Z}^d\setminus S$. Consider the Hilbert space of complex-valued sequences
\begin{equation*}
\ell^2_p:=\left\{q=(q_{\textbf{j}})_{\textbf{j}\in\mathbb{Z}_1^d}\bigg{|}\parallel q \parallel_p^2:=\sum_{\textbf{j}\in\mathbb{Z}_1^d}|q_\textbf{j}|^2
|\textbf{j}|_2^{2p}<+\infty\right\}
\end{equation*}
with $p>d$ and
\begin{equation*}
|\textbf{j}|_2=\sqrt{|j_1|^2+\cdots+|j_d|^2}, \quad \textbf{j}=(j_1,\dots,j_d)\in\mathbb{Z}_1^d,
\end{equation*}and the symplectic phase space
\begin{equation*}
(x,y,z)\in\mathbb{T}_s^n\times\mathbb{C}^n\times\ell_{b,p}^2:=\mathcal{P}^p,\qquad z:=(q,\bar q)\in\ell_{b,p}^2:=\ell_p^2\times\ell^{2}_p,
\end{equation*}
where $\mathbb{T}_s^n$ is the complex open $s$-neighbourhood of the $n$-torus $\mathbb{T}^n:=\mathbb{R}^n/(2\pi\mathbb{Z})^n$, equipped with the canonic symplectic structure:
\begin{equation*}
\sum_{i=1}^{n}dy_i\wedge dx_i+\sqrt{-1}\sum_{\textbf{j}\in\mathbb{Z}_1^d}dq_{\textbf{j}}\wedge d\bar{q}_{\textbf{j}}.
\end{equation*}
Let
\begin{equation*}
D(s,r_1,r_2)=\left\{(x,y,z)\in \mathcal{P}^p \big{|}\;\parallel\mbox{Im}\ x\parallel< s,\parallel y\parallel < r_1^2,\parallel z\parallel_p< r_2\right\},
\end{equation*}
where $\parallel\cdot\parallel$ denote the sup-norm for
complex vectors and
\begin{equation*}
\parallel z\parallel_{p}=\parallel q\parallel_{p}+\parallel\bar q\parallel_p,\qquad\mbox{with}\ z=(q,\bar q).
\end{equation*}
 Any analytic function $W:D(s,r_1,r_2)\rightarrow\mathbb{C}$ can be developed in a totally convergent power series:
\begin{equation*}
W(x,y,z)=\sum_{\alpha\in\mathbb{N}^n,\beta\in\mathbb{N}^{{\mathbb{Z}}_1^d}}
W^{\alpha\beta}(x)y^{\alpha}z^{\beta}.
\end{equation*}Note that there is a multilinear, symmetric, bounded map
\begin{equation*}
\widetilde{W^{\alpha\beta}(x)}\in\mathcal{L}\left(\overbrace{\mathbb{C}^n
\times\dots\times\mathbb{C}^n}^{|\alpha|-times}\times\overbrace{
\ell_{b,p}^2\times\dots\times\ell^{2}_{b,p}}^{|\beta|-times},\mathbb{C}\right),
\end{equation*}
such that
$$
\widetilde {W^{\alpha\beta}(x)}(\overbrace{y,\dots,y}^{|\alpha|-times},
\overbrace{z.\dots,z}^{|\beta|-times})=W^{\alpha\beta}(x)y^{\alpha}z^{\beta},$$
where$$|\alpha|=\sum_{i=1}^{n}|\alpha_i|,$$
and
$$|\beta|=\sum_{{\textbf{j}}\in{{\mathbb{Z}}_1^d}}|\beta_{\textbf{j}}|,$$ and $|\cdot|$ denotes the 1-norm here and below.

We will study the Hamiltonian system
\begin{equation*}
(\dot{x},\dot{y},\dot{z})=X_W(x,y,z),
\end{equation*}
where $X_W$ is the Hamiltonian vector field of $W$,
\begin{equation*}
X_W=(W_y,-W_x,\sqrt{-1}JW_z),
\end{equation*}
and
\begin{equation*}
J:=\left(\begin{array}{cc}
0&I\\-I&0
\end{array}\right).
\end{equation*}

\begin{definition}\label{112901}
Let $D(s)=\{ x \in \mathbb{T}_s^n | \parallel \mbox{Im} x\parallel < s\}$. Consider a function $W(x;\xi):D(s)\times\Pi\rightarrow \mathbb{C}$
is analytic in the variable $x\in D(s)$ and $C^1$-smooth in the parameter $\xi\in\Pi$ in the
Whitney's sense\footnote{In the whole of this paper, the derivatives
with respect to the parameter $\xi\in\Pi$ are understood in the
sense of Whitney.}, and the Fourier
series of $W(x;\xi)$ is given by
$$W(x;\xi)=\sum_{k\in \mathbb Z^n}\widehat
W(k;\xi)e^{\sqrt{-1}\langle k,x\rangle},$$where $$\widehat
W(k;\xi):=\frac1{(2\pi)^n}\int_{\mathbb{T}^n}W(x;\xi)e^{-\sqrt{-1}\langle
k,x\rangle}dx$$ is the $k$-th Fourier coefficient of
$W(x;\xi)$, and $\langle\cdot,\cdot\rangle$ denotes the usual inner
product, i.e.
\begin{equation*}
\langle k,x\rangle=\sum_{i=1}^nk_ix_i.
\end{equation*}
Then define the norm $\parallel\cdot\parallel_{D(s)\times\Pi}$ of $W(x;\xi)$ by
\begin{equation}\label{092704}
\parallel W\parallel_{D(s)\times\Pi}=\sup_{\xi\in\Pi,\textbf{j}\in\mathbb{Z}^d}{\sum_{k\in\mathbb{Z}^n}
\left(|\widehat{W}(k;\xi)|+|\partial_{\xi_{\textbf{j}}}
\widehat{W}(k;\xi)|\right)e^{|k|s}}.
\end{equation}
\end{definition}
\begin{definition}\label{112902}Let
\begin{equation*}
D(s,r)=\{(x,y)\in\mathbb{T}_s^n\times\mathbb{C}^n|\ \parallel\mbox{Im}\
x\parallel < s,\ \parallel y\parallel < r^2\}.
\end{equation*}Consider a function $W(x,y;\xi):D(s,r)\times\Pi\rightarrow
\mathbb{C}$ is analytic in the variable $(x,y)\in D(s,r)$ and
$C^1$-smooth in the parameter $\xi\in\Pi$ with the following form
\begin{equation*}
W(x,y;\xi)=\sum_{\alpha\in\mathbb{N}^n}W^{\alpha}(x;\xi)y^{\alpha}.
\end{equation*}Then
define the norm $\parallel\cdot\parallel_{D(s,r)\times\Pi}$ of $W(x,y;\xi)$
by \begin{equation}\label{092703}
\parallel W\parallel_{D(s,r)\times\Pi}=\sum_{\alpha\in\mathbb{N}^n}
|||\widetilde{\mathcal{W}^{\alpha}}|||\;r^{2|\alpha|},
\end{equation}
where $\mathcal{W}^{\alpha}=\parallel W^{\alpha}(x;\xi)\parallel_{D(s)\times\Pi}$,
 $\widetilde{{\mathcal W}^{\alpha}}\in\mathcal{L}(\overbrace{\mathbb{C}^n\times\dots\times\mathbb{C}^n}^{|\alpha|-times},\mathbb{C})$$ $ is an $|\alpha|$-linear symmetric bounded map such that $$
\widetilde {\mathcal W^{\alpha}}(\overbrace{y,\dots,y}^{|\alpha|-times})=\mathcal{W}^{\alpha}y^{\alpha},$$and $|||\cdot|||$ is the operator norm of multilinear symmetric bounded maps.
\end{definition}

 \begin{definition}\label{021002}
Consider a function
$W(x,y,z;\xi):D(s,r,r)\times\Pi\rightarrow\mathbb{C}$ is
analytic in the variable $(x,y,z)\in D(s,r,r)$ and
$C^1$-smooth in the parameter $\xi\in\Pi$ with the following form
$$W(x,y,z;\xi)=\sum_{\beta\in\mathbb{N}^{{\mathbb{Z}_1^d}}}
W^{\beta}(x,y;\xi)z^{\beta}.$$
Define the modulus $\lfloor W\rceil_{D(s,r)\times\Pi}(z)$ of
$W(x,y,z;\xi)$ by
\begin{equation}\label{081602}\lfloor W\rceil_{D(s,r)\times\Pi}(z):=
\sum_{\beta\in\mathbb{N}^{{\mathbb{Z}_1^d}}}\parallel W^{\beta}\parallel_{D(s,r)\times\Pi}\;
z^{\beta}.\end{equation}
\end{definition}

For $h\geq 1$, let
\begin{equation}\label{091401}\parallel (z^{h})\parallel_{p,d}:=
\frac{1}{h}\sum_{i=1}^{h}\parallel z^{(1)}\parallel_d\cdots\parallel z^{(i-1)}\parallel_d\parallel
 z^{(i)}\parallel_p\parallel z^{(i+1)}\parallel_d\cdots\parallel z^{(h)}\parallel_d.
\end{equation}
\begin{remark}
For $h=1$, it is easy to see that
\begin{equation}
\parallel(z^{h})\parallel_{p,d}=\parallel z\parallel_p.
\end{equation}
\end{remark}
\begin{definition}\label{071803}($p$-tame norm for a homogeneous Hamiltonian) \\ Let $$W(x,y,z;\xi):=W_h(x,y,z;\xi)=\sum_{\beta\in\mathbb{N}^{{\mathbb{Z}_1^d}},|\beta|=h}
W_h^{\beta}(x,y;\xi)z^{\beta}$$ be a function
is analytic in the variable $(x,y,z)\in D(s,r,r)$ and
$C^1$-smooth in the parameter $\xi\in\Pi$.
Define the $p$-tame operator norm for $W_z$ by
\begin{eqnarray}
&&\nonumber|||{W_z}|||_{p,D(s,r)\times\Pi}^{T}\\&: =&\sup_{0\neq z^{(i)}\in \ell
^2_{b,p},1\leq i\leq h-1}
\frac{\parallel{\lfloor{\widetilde{JW_z}}\rceil}_{D(s,r)\times\Pi}(z^{(1)},\dots,z^{(h-1)})\parallel_
{p+2}} {\parallel(z^{h-1})\parallel_{p,d}}\label{091402}, \quad h\geq 2,
\end{eqnarray}
and
\begin{equation}
|||{W_z}|||_{p,D(s,r)\times\Pi}^{T}: =\sup_{0\neq z\in \ell
^2_{b,p}}
{\parallel{\lfloor{\widetilde{JW_z}}\rceil}_{D(s,r)\times\Pi}(z)\parallel_
{p+2}}\label{091402'}, \quad h=0,1,
\end{equation}
define the $d$-operator norm for $W_z$ by
\begin{eqnarray}
&&\nonumber|||{W_z}|||_{d,D(s,r)\times\Pi}\\&: =&\sup_{0\neq z^{(j)}\in \ell
^2_{b,d},1\leq i\leq h-1}
\frac{\parallel{\lfloor{\widetilde{JW_z}}\rceil}_{D(s,r)\times\Pi}(z^{(1)},\dots,z^{(h-1)})\parallel_
{d}} {\parallel(z^{h-1})\parallel_{d,d}},\label{122401}\quad h\geq 2,\end{eqnarray}
and
\begin{equation}
|||{W_z}|||_{d,D(s,r)\times\Pi}: =\sup_{0\neq z\in \ell
^2_{b,d}}
{\parallel{\lfloor{\widetilde{JW_z}}\rceil}_{D(s,r)\times\Pi}(z)\parallel_
{d}}\label{122401'}, \quad h=0,1,
\end{equation}
 and define the operator norm for $W_v$ ($v=x$ or $y$) by
\begin{eqnarray} &&\nonumber|||W_{v}|||_{D(s,r)\times\Pi}\\&:=&\label{091404}\sup_{0\neq
z^{(i)}\in \ell ^2_{b,d},1\leq i\leq h}
\frac{\parallel{\lfloor{\widetilde{W_{v}}}\rceil}_{D(s,r)\times\Pi}(z^{(1)},\dots,z^{(h)})\parallel}
{\parallel(z^{h})\parallel_{d,d}},\quad h\geq 1,
\end{eqnarray}
and
\begin{equation}
|||{W_v}|||_{D(s,r)\times\Pi}: =\sup_{0\neq z\in \ell
^2_{b,d}}
{\parallel{\lfloor{\widetilde{W_v}}\rceil}_{D(s,r)\times\Pi}(z)\parallel}\label{091404'}, \quad h=0.
\end{equation}Finally define the $p$-tame norm of the Hamiltonian vector field $X_W$ as
follows,
\begin{eqnarray}\label{051703}
&&\nonumber|||X_W|||_{p,D(s,r,r)\times\Pi}^T \\&:=&|||{W_y}|||_{D(s,r,r)\times\Pi}
+\frac1{r^2}|||{W_x}|||_{D(s,r,r)\times\Pi}
+\frac1r|||W_z|||_{p,D(s,r,r)\times \Pi}^T,
\end{eqnarray}where
\begin{eqnarray}\label{091405}
|||{W_v}|||_{D(s,r,r)\times\Pi}:=|||{W_v}|||_{D(s,r)\times\Pi}r^{h},\qquad v=x\ \mbox{or}\ y,
\end{eqnarray}
and
\begin{equation}\label{091403}
|||{W_z}|||_{p,D(s,r,r)\times\Pi}^T:=\max
\left\{|||{W_z}|||_{p,D(s,r)\times\Pi}^T,|||{W_z}|||_{d,D(s,r)\times\Pi}\right\}r^{h-1}.
\end{equation}
\end{definition}
\begin{remark}
In view of (\ref{091402}), $\lfloor{{JW_z}}\rceil_{D(s,r)\times\Pi}$ is required as a bounded map form $\ell^{2}_{b,p}$ to $\ell^{2}_{b,p+2}$ instead of a bounded map form $\ell^{2}_{b,p}$ to $\ell^{2}_{b,p}$ as in \cite{CLY}. This regularity is necessary to guarantee KAM iterative procedure work for the spacial dimension $d\geq 2$ (not necessary for $d=1$).
\end{remark}
\begin{remark}
Based on (\ref{091402}) and (\ref{091404})
 in Definition \ref{071803}, for each
$(x,y,z)\in\mathcal{P}^p$ and $\xi\in\Pi$, the following estimates hold
\begin{eqnarray}
&&\nonumber\parallel(W_h)_z(x,y,z;\xi)\parallel_p\\
&\leq&\parallel(W_h)_z(x,y,z;\xi)\parallel_{p+2}\nonumber\\&\leq&\label{011602}
|||{(W_h)_z}|||_{p,D(s,r)\times\Pi}^{T}\parallel z\parallel_p\parallel z\parallel_d^{\max\{h-2,0\}},\quad h\geq 2,
\end{eqnarray}
and
\begin{equation}\label{012202} ||(W_h)_v(x,y,z;\xi)||\leq
|||{(W_h)_v}|||_{D(s,r)\times\Pi}\parallel z\parallel_d^{h},\qquad h\geq1.
\end{equation}
\end{remark}
\begin{definition}\label{080204}($p$-tame norm for a general Hamiltonian)\\ Let $W(x,y,z;\xi)=\sum_{h\geq
0}W_{h}(x,y,z;\xi)$ be a Hamiltonian analytic in the variable $(x,y,z)\in
D(s,r,r)$ and $C^1$-smooth in the parameter $\xi\in\Pi$, where
$$W_{h}(x,y,z;\xi)=\sum_{\beta\in\mathbb{N}^{{\mathbb{Z}_1^d}},|\beta|=h}W_{h}^{\beta}(x,y;\xi)
z^{\beta}.$$ Then define the $p$-tame norm of the Hamiltonian
vector field $X_W$ by
\begin{equation}\label{102201}
|||X_W|||_{p,D(s,r,r)\times\Pi}^T:=\sum_{h\geq 0}
|||X_{W_{h}}|||_{p,D(s,r,r)\times\Pi}^T. \end{equation} Moreover, we say that
a Hamiltonian vector field $X_W$ (or a Hamiltonian $W(x,y,z;\xi)$)
has ${p}$-tame property on the domain $D(s,r,r)\times\Pi$, if and
only if $$ |||X_W|||_{p,D(s,r,r)\times\Pi}^{T}<\infty. $$
\end{definition}

\section{The abstract results}\label{section}

\begin{theorem}\label{T1}(Normal form of order 2)
Consider a perturbation of the integrable Hamiltonian
\begin{equation}\label{100803} H(x,y,q,\bar q;\xi)=N(y,q,\bar q;\xi)+R(x,y,q,\bar
q;\xi)
\end{equation}defined on the domain $D(s_0,r_0,r_0)\times\Pi$ with
$s_0,r_0\in(0,1]$, where
$$N(y,q,\bar q;\xi)=\sum_{i=1}^{n}\omega_i(\xi)y_i+\sum_{\textbf{j}\in\mathbb{Z}_1^d}\Omega_{\textbf{j}}
(\xi)q_{\textbf{j}}\bar
q_{\textbf{j}}$$ is a family of parameter dependent integrable Hamiltonian and
$$R(x,y,q,\bar q;\xi)=\sum_{\alpha\in\mathbb{N}^n,\beta,\gamma\in\mathbb{N}^{\mathbb{Z}_1^d}}R^{\alpha\beta\gamma}(x;\xi)y^{\alpha}q^{\beta}\bar q^{\gamma}$$ is the perturbation. Suppose the tangent
frequency and normal frequency satisfy the following assumption:

\begin{enumerate}
\item Frequency Asymptotic.
\begin{equation}\label{012801}
\omega_{i}(\xi)= |{\textbf{j}}_i|_2^2+ \xi_{{\textbf{j}}_i},\qquad   1\leq i\leq n,
\end{equation}
and
\begin{equation}\label{012802}
\Omega_{\textbf{j}}(\xi)= |\textbf{j}|_2^2+\xi_{\textbf{j}}\qquad \mbox{for} \ \textbf{j}\in\mathbb{Z}_1^{d},
\end{equation}
where $$\xi=((\xi_{\textbf{j}})_{\textbf{j}\in\mathbb{Z}^d})\in\Pi\subset\mathbb{R}^{\mathbb{Z}^d}.$$

\item Tame Property  and  smallness conditions.
The perturbation $R(x,y,q,\bar q;\xi)$ has $p$-tame property on the
domain $D(s_0,r_0,r_0)\times\Pi$ and satisfies the small assumption:
\begin{equation*}
 \varepsilon:=|||X_{R}|||^T_{p,D(s_0,r_0,r_0)\times\Pi}\leq
 \eta^{12}\epsilon, \qquad \mbox{ for some }\ \eta\in(0,1),
\end{equation*}where $\epsilon$ is a positive constant
depending on $s_0,r_0$ and $n$.

\item Spacial form of perturbation.
The perturbation $R(x,y,q,\bar q;\xi)$ is taken from a special class of analytic functions
\begin{equation}\label{04041}
\mathcal{A}=\left\{R:R=\sum_{k\in\mathbb{Z}^n,\alpha\in\mathbb{N}^n,\beta,\gamma\in\mathbb{N}^{\mathbb{Z}_1^d}}\widehat{R(k;\xi)}y^{\alpha}q^{\beta}\bar q^{\gamma}\right\}
,\end{equation}
where $k,\alpha,\beta$ has the following relation
\begin{equation}\label{04051}
\sum_{i=1}^nk_i\textbf{j}_i+\sum_{\textbf{j}\in\mathbb{Z}_1^d}(\beta_{\textbf{j}}-\gamma_{\textbf{j}})\textbf{j}=0.
\end{equation}
\end{enumerate}
Then there exists a subset
$\Pi_{\eta}\subset\Pi$ with the estimate
\begin{equation*}
\mbox{Meas}\ \Pi_{\eta}\geq(\mbox{Meas}\ \Pi)(1-O(\eta)).
\end{equation*}For each $\xi\in\Pi_{\eta}$, there is a symplectic map
$$\Psi: D(s_0/2,r_0/2,r_0/2)\rightarrow D(s_0,r_0,r_0),$$ such that
\begin{equation}\label{081601}\breve H(x,y,q,\bar q;\xi):=H\circ\Psi=
\breve N(y,q,\bar q;\xi)+ \breve R(x,y,q,\bar q;\xi),
\end{equation}where
\begin{equation}\label{082903}
\breve N(y,q,\bar
q;\xi)=\sum_{i=1}^{n}\breve{\omega}_i(\xi)y_i+\sum_{\textbf{j}\in\mathbb{Z}_1^d}\breve{\Omega}_{\textbf{j}}(\xi)q_{\textbf{j}}\bar
q_{\textbf{j}}
\end{equation}
and
\begin{equation}\label{082902}
\breve R(x,y,q,\bar
q;\xi)=\sum_{\alpha\in\mathbb{N}^n,\beta,\gamma\in\mathbb{N}^{\mathbb{Z}_1^d}
,2|\alpha|+|\beta|+|\gamma|\geq
3}\breve
R^{\alpha\beta\gamma}(x;\xi)y^{\alpha}q^{\beta}\bar{q}^{\gamma}.
\end{equation}
Moreover, the following estimates hold:\\
(1) for each $\xi\in\Pi_{\eta},$ the symplectic map
$\Psi:D(s_0/2,r_0/2,r_0/2)\rightarrow D(s_0,r_0,r_0)$
satisfies
\begin{equation}\label{080101}
\parallel\Psi-id\parallel_{p,D(s_0/2,r_0/2,r_0/2)}\leq c\eta^6\epsilon,
\end{equation}
where
\begin{equation}\label{081102}
\parallel\Psi-id\parallel_{p,D(s_0/2,r_0/2,r_0/2)}=\sup_{w\in
D(s_0/2,r_0/2,r_0/2)}\parallel(\Psi-id)w\parallel_{\mathcal{P}^p,D(s_0,r_0,r_0)},
\end{equation}
 moreover,
\begin{equation}\label{080102}
|||D\Psi-Id|||_{p,D(s_0/2,r_0/2,r_0/2)}\leq c\eta^6\epsilon,
\end{equation}
where on the left-hand side hand we use the operator
norm\footnote{where $id$ denotes the identity map from $\mathcal
P^p\to\mathcal P^p$ and $Id$ denotes its tangent map. }
\begin{equation*}
|||D\Psi-Id|||_{p,D(s_0/2,r_0/2,r_0/2)}=\sup_{0\neq w\in
D(s_0/2,r_0/2,r_0/2) }\frac{\parallel(D\Psi-Id) w\parallel_{\mathcal{P}^p,D(s_0,r_0,r_0) }}{\parallel
w\parallel_{\mathcal{P}^p,D(s_0/2,r_0/2,r_0/2)}};
\end{equation*}
\\ (2) the frequencies $\breve \omega(\xi)$ and $\breve
\Omega(\xi )$ satisfy
\begin{equation}\label{080103}
\parallel\breve\omega(\xi)-\omega(\xi)\parallel+\sup_{\textbf{j}\in\mathbb{Z}^d}\parallel\partial_{\xi_{\textbf{j}}}(\breve\omega(\xi)-\omega(\xi))\parallel\leq
{c\eta^8\epsilon},\end{equation}and
\begin{equation}\label{080104}
\parallel\breve\Omega(\xi)-\Omega(\xi)\parallel_{-2}+\sup_{\textbf{j}\in\mathbb{Z}^d
}\parallel\partial_{\xi_{\textbf{j}}}(\breve\Omega(\xi)-\Omega(\xi))\parallel_{-2}\leq
{c{\eta^8}\epsilon},
\end{equation}
where
\begin{equation}
\parallel \Omega(\xi)=(\Omega_{\textbf{j}}(\xi))_{{\textbf{j}}\in\mathbb{Z}^d_1}\parallel_{-2}:=\sup_{\textbf{j}\in\mathbb{Z}^d_1}|\Omega_{\textbf{j}}(\xi)|\textbf{j}|_2^2|;
\end{equation}
(3) the Hamiltonian vector field $X_{\breve R}$ of the new perturbed
Hamiltonian $\breve R(x,y,q,\bar q;\xi)$ satisfies
\begin{equation}\label{080105}
 |||X_{\breve R}|||^T_{p,D(s_0/2,r_0/2,r_0/2)\times\Pi_{\eta}}\leq
 \varepsilon(1+{c\eta^6\epsilon}),
\end{equation}where $c>0$ is a constant depending on $s_0,r_0$ and $n$.
\end{theorem}
\begin{remark}
This theorem is parallel to Theorem in \cite{CLY} and is essentially due to a standard KAM proof. The same as in \cite{CLY}, the tame property (\ref{080105}) of $X_{\breve R}$ can be verified explicitly in view of Lemmas \ref{021102}-\ref{081503}. Moreover, as a corollary of this theorem, the existence and long time stability can be obtained directly.
\end{remark}
Given a large $\mathcal{N}\in\mathbb{N}$, split the normal frequency
$\breve{\Omega}(\xi)$ and normal variable $(q,\bar q)$ into two parts
respectively, i.e.
$$\breve{\Omega}(\xi)=(\tilde {\Omega}(\xi),\hat{\Omega}(\xi)),\qquad q=(\tilde
q,\hat q),\qquad \bar q=(\tilde{\bar q},\hat {\bar q}),$$ where
$$\tilde
{\Omega}(\xi)=(\breve{\Omega}_{\textbf{j}}(\xi))_{|\textbf{j}|_2\leq \mathcal{N}},\quad
\tilde q=(q_{\textbf{j}})_{|\textbf{j}|_2\leq \mathcal{N}},\quad \tilde{\bar q}=(\bar
q_{\textbf{j}})_{|\textbf{j}|_2\leq\mathcal{N}}$$ are the low frequencies  and
$$
\hat{\Omega}(\xi)=(\breve{\Omega}_{\textbf{j}}(\xi))_{|\textbf{j}|_2>\mathcal{N}},\quad
\hat q=(q_{\textbf{j}})_{|\textbf{j}|_2>\mathcal{N}},\quad \hat{\bar
q}=(\bar q_{\textbf{j}})_{|\textbf{j}|_2>\mathcal{N}}$$
 are the high frequencies. Given $0<\tilde \eta<1$,
and $\tau> 2n+5$, if the frequencies $\breve{\omega}(\xi)$ and
$\breve{\Omega}(\xi)$ satisfy the following inequalities
\begin{equation}\label{090302} \left|\langle
k,\breve{\omega}(\xi)\rangle+\langle \tilde
l,\tilde{\Omega}(\xi)\rangle+\langle
\hat{l},\hat{\Omega}(\xi)\rangle\right|\geq\frac{\tilde\eta}{4^{3\mathcal{M}}(|
k|+1)^{\tau}C(\mathcal{N},\tilde l)},
\end{equation}
with $$|k|+|\tilde l|+|\hat l|\neq 0,\quad |\tilde l|+|\hat
l|\leq \mathcal{M}+2,\quad |\hat l|\leq 2,$$ where
\begin{equation}\label{121206}
C(\mathcal{N},\tilde l)=\mathcal{N}^{3(|\tilde l|+4)^2},
\end{equation}
then we call that the
frequencies $\breve{\omega}(\xi)$ and $\breve{\Omega}(\xi)$ are
$(\tilde\eta,\mathcal{N},\mathcal{M})$-non-resonant.

\begin{remark}
Denote the resonant set $\mathcal{R}_{k\tilde l\hat l}$ by
\begin{equation}
\mathcal{R}_{k\tilde l\hat l}=\left\{\xi\in\Pi_{\eta}\bigg| \left|\langle
k,\breve{\omega}(\xi)\rangle+\langle \tilde
l,\tilde{\Omega}(\xi)\rangle+\langle
\hat{l},\hat{\Omega}(\xi)\rangle\right|<\frac{\tilde\eta}{4^{3\mathcal{M}}(|
k|+1)^{\tau}C(\mathcal{N},\tilde l)}\right\},
\end{equation}
where $\Pi_{\eta}$ is given in Theorem \ref{T1}, and denote
\begin{equation}\label{0912081}
\mathcal{R}=\bigcup_{|k|+|\tilde l|+|\hat l|\neq 0,\ |\tilde l|+|\hat
l|\leq \mathcal{M}+2,\ |\hat l|\leq 2}\mathcal{R}_{k\tilde l\hat l}.
\end{equation}
Then for each
\begin{equation}\label{032501}\xi\in\tilde \Pi:=\Pi_{\eta}\setminus \mathcal{R},
\end{equation} the
frequencies $\breve{\omega}(\xi)$ and $\breve{\Omega}(\xi)$ are
$(\tilde\eta,\mathcal{N},\mathcal{M})$-non-resonant.
\end{remark}
\begin{theorem}\label{thm7.1} (Partial normal form of order $\mathcal{M}+2$)
Consider the normal form of order 2 $$\breve H(x,y,q,\bar
q;\xi)=\breve N(y,q,\bar q;\xi)+\breve R(x,y,q,\bar q;\xi)$$
obtained in Theorem \ref{T1}. Suppose $\xi\in\tilde\Pi$, which is defined in (\ref{032501}), for some positive integers
$\mathcal{N},\mathcal{M}$ and $0<\tilde \eta<1$, there exist a small $\rho_0>0$
depending on
$s_0,r_0,n,\tilde\eta, \mathcal{N}$ and $\mathcal{M}$, and for each $0<\rho<\rho_0$, there is a symplectic
map
$$\Phi: D(s_0/4,4\rho,4\rho)\rightarrow D(s_0/2,5\rho,5\rho),$$ such
that
\begin{equation}\label{020302}
\breve H\circ\Phi={\breve N}(y,q,\bar q;\xi)+{Z}(y,q,\bar q;\xi)+{
P}(x,y,q,\bar q;\xi)+{ Q}(x,y,q,\bar q;\xi)
\end{equation}
is a partial normal form of order $\mathcal{M}+2$, where
\begin{eqnarray*}
{ Z}(y,q,\bar q;\xi)&=&\sum_{4\leq 2|\alpha|+2|\beta|+2|\mu|\leq
\mathcal{M}+2,|\mu|\leq1}{
Z}^{\alpha\beta\beta\mu\mu}(\xi)y^{\alpha}\tilde{q}^{\beta}\tilde{\bar
q}^{\beta}\hat{q}^{\mu}\hat{\bar{q}}^{\mu}
\end{eqnarray*}is the integrable term depending only on $y$ and $I_{\textbf{j}}=|q_{\textbf{j}}|^2,{\textbf{j}}\in\mathbb{Z}_1^d$, and where
\begin{eqnarray*}  
{P}(x,y,q,\bar
q;\xi)&=&\sum_{2|\alpha|+|\beta|+|\gamma|+|\mu|+|\nu|\geq
\mathcal{M}+3,|\mu|+|\nu|\leq 2} { P}^{\alpha
\beta\gamma\mu\nu}(x;\xi)y^{\alpha}{\tilde q}^{\beta}{\tilde{\bar
q}}^{\gamma}{\hat q}^{\mu}{\hat{\bar q}^{\nu}},
\end{eqnarray*}and
\begin{equation*}
{ Q}(x,y,q,\bar q;\xi)=\sum_{|\mu|+|\nu|\geq 3}{ Q}^{\alpha
\beta\gamma\mu\nu}(x;\xi)y^{\alpha}{\tilde q}^{\beta}{\tilde{\bar
q}}^{\gamma}{\hat q}^{\mu}{\hat{\bar q}^{\nu}}.
\end{equation*}
Moreover, we have the following estimates:\\
(1) the symplectic map $\Phi$ satisfies
\begin{equation}\label{091812}
\parallel\Phi-id\parallel_{p,D(s_0/4,4\rho,4\rho)}\leq \frac{c\mathcal{N}^{294}\rho}{\tilde\eta^2},
\end{equation}
and
\begin{equation}
|||D\Phi-Id|||_{p,D(s_0/4,4\rho,4\rho)}\leq \frac{c\mathcal{N}^{294}}{\tilde\eta^2};
\end{equation}
\\
(2) the Hamiltonian vector fields $X_Z,X_P$ and $X_Q$ satisfy
\begin{equation*}
|||X_{{ Z}}|||^T_{p,D(s_0/4,4\rho,4\rho)\times\tilde{\Pi}}\leq
c\rho
\left(\frac1{\tilde\eta^2}\mathcal{N}^{6(\mathcal{M}+6)^2}\rho\right),
\end{equation*}
\begin{equation}\label{091304}
|||X_{{ P}}|||^T_{p,D(s_0/4,4\rho,4\rho)\times\tilde{\Pi}}\leq
c
\rho\left(\frac1{\tilde\eta^2}\mathcal{N}^{6(\mathcal{M}+7)^2}\rho\right)^{\mathcal{M}},\qquad
 \end{equation}and
\begin{equation*}
|||X_{{ Q}}|||^T_{p,D(s_0/4,4\rho,4\rho)\times\tilde{\Pi}}\leq
c \rho,
\end{equation*}
where $c>0$ is a constant depending on $s_0,r_0,n$ and
$\mathcal{M}$.
\end{theorem}

Based on the partial normal form of order $\mathcal{M}+2$ and $p$-tame property, we obtain the long time stability of KAM tori as follows:
\begin{theorem}\label{T3}(The long time stability of KAM tori) Based on the partial normal form (\ref{020302}), for any $p\geq
24(\mathcal{M}+7)^{4}+1$ and $0<\delta<\rho$, the KAM tori ${\mathcal T}$ are stable in
long time, i.e. if $w(t)$ is a solution of Hamiltonian vector field
$X_{{H}}$ with the initial datum $w(0)=( w_x(0), w_y(0),
w_q(0),w_{\bar q}(0))$ satisfying
\begin{equation*}
d_{p}( w(0),{\mathcal{T}})\leq \delta,
\end{equation*}
then
\begin{equation}\label{111905}
d_{p}(w(t),{\mathcal{T}})\leq 2\delta,
\qquad\mbox{for all}\ |t|\leq \delta^ {-\mathcal{M}}.
\end{equation}
\end{theorem}
\section{Proof of Theorem \ref{T4}}
\begin{proof}

{Firstly, write equation (\ref{26}) as an infinite dimensional  Hamiltonian system.}

Here we assume that the operator $A=-\triangle+M_{\xi}$ with periodic boundary conditions has eigenvalues $\lambda_\textbf{j}$ satisfying
\begin{eqnarray}
\lambda_\textbf{j}=|\textbf{j}|_2^2+\xi_\textbf{j},\qquad \textbf{j}\in\mathbb{Z}^d,
\end{eqnarray}
and the corresponding eigenfunctions $$\phi_\textbf{j}(x)=\frac1{(2\pi)^{d/2}}e^{\langle \textbf{j},x\rangle}$$ form a basis in the domain of the operator.

Introducing $v=u_t$, (\ref{26}) reads
\begin{eqnarray}
u_t&=&v,\nonumber\\
v_t&=&-A^2u-\varepsilon f(u).
\end{eqnarray}
Letting
\begin{equation}
q=\frac1{\sqrt{2}}A^{\frac12}u-\sqrt{-1}\frac1{\sqrt{2}}A^{-\frac12}v,
\end{equation}
we obtain
\begin{equation}\label{031301}
-\sqrt{-1}q_t=Aq+\frac{\varepsilon}{\sqrt{2}}A^{-\frac12}f\left(A^{-\frac12}\left(\frac{q+\bar q}{\sqrt{2}}\right)\right).
\end{equation}
Equation (\ref{031301}) can be rewritten as the Hamiltonian equations
\begin{equation}\label{031302}
q_t=\sqrt{-1}\frac{\partial H}{\partial \bar q},
\end{equation}
and the corresponding Hamiltonian is
\begin{equation}
H=\frac12(Aq,q)+\varepsilon \int_{\mathbb{T}^d}g\left(A^{-\frac12}\left(\frac{q+\bar q}{\sqrt{2}}\right)\right)dx,
\end{equation}
where $(\cdot,\cdot)$ denotes the inner product in $L^2$ and $g$ is a primitive of $f$.

Let
\begin{equation}
q(x)=\sum_{\textbf{j}\in\mathbb{Z}^d}q_\textbf{j}\phi_\textbf{j}(x).
\end{equation}
Thus system (\ref{031302}) is equivalent to the lattice Hamiltonian equations
\begin{equation}\label{031401}
\dot{q}_\textbf{j}=\sqrt{-1}\left(\lambda_\textbf{j}q_\textbf{j}+\varepsilon\frac{\partial G}{\partial \bar{q}_\textbf{j}}\right),\qquad G(q,\bar q):=\int_{\mathbb{T}^d}g\left(\sum_{\textbf{j}\in\mathbb{Z}^d}\frac{q_\textbf{j}\phi_\textbf{j}+\bar q_\textbf{j}\bar\phi_\textbf{j}}{\sqrt{2\lambda_\textbf{j}}}\right)dx
\end{equation}
with the corresponding Hamiltonian function
\begin{equation}
H(q,\bar q)=\sum_{\textbf{j}\in\mathbb{Z}^d}\lambda_\textbf{j}q_\textbf{j}\bar q_\textbf{j}+\varepsilon G(q,\bar q).
\end{equation}
Since $f(u)$ is real analytic in $u$, $g(q,\bar q)$ is real analytic in $q,\bar q$. Making use of
\begin{equation}
q(x)=\sum_{\textbf{j}\in\mathbb{Z}^d}q_\textbf{j}\phi_\textbf{j}(x)
\end{equation}
again, we may rewrite $g(q,\bar q)$ as follows
\begin{equation}
g(q,\bar q)=\sum_{\alpha,\beta}g^{\alpha\beta}q^{\alpha}\bar q^{\beta}\phi^{\alpha}\bar \phi^{\beta}.
\end{equation}
Hence,
\begin{equation}\label{01019}
G(q,\bar q):=\int_{\mathbb{T}^d}g\left(\sum_{\textbf{j}\in\mathbb{Z}^d}\frac{q_\textbf{j}\phi_\textbf{j}+\bar q_\textbf{j}\bar\phi_\textbf{j}}{\sqrt{2\lambda_\textbf{j}}}\right)dx=\sum_{\alpha,\beta}G^{\alpha\beta}q^{\alpha}\bar q^{\beta},
\end{equation}
where
\begin{equation}\label{11112}
G^{\alpha\beta}=0,\qquad \mbox{if}\quad
\sum_{\textbf{j}\in\mathbb{Z}^d}(\alpha_\textbf{j}-\beta_\textbf{j})\textbf{j}\neq 0.
\end{equation}
To simply the proof, we assume $f(u)=u^3$ without loss of generality. following example 3.2 in \cite{BG}, we have
\begin{equation}
\parallel X_{G(q^{(1)},q^{(2)},q^{(3)})}\parallel_p \leq c_p \parallel z^3 \parallel{p,d},
\end{equation}
and
\begin{equation}\label{00001}
\parallel X_{G(q^{(1)},q^{(2)},q^{(3)})}\parallel_d \leq c_p \parallel z^3 \parallel{d,d}.
\end{equation}
Furthermore, as in \cite{P1}, the perturbation $G(q,\bar{q})$ is more regular in the following sense
\begin{equation}\label{00002}
\parallel X_{G(q^{(1)},q^{(2)},q^{(3)})}\parallel_{p+2} \leq c_p \parallel z^3 \parallel{p,d}.
\end{equation}

As in \cite{GY20061}, the perturbation $G(q,\bar q)$ in (\ref{031401}) has the following regularity property.
\begin{lemma}
For any fixed $p>d/2$, the gradient $G_{\bar q}$ is a map in a neighbourhood of the origin with
\begin{equation}
\parallel G_{\bar q}\parallel_{p+2}\leq c\parallel q\parallel_p^3.
\end{equation}
\end{lemma}

Next we introduce standard action-angle variables
\begin{equation}
(x,y)=((x_1,\dots,x_n),(y_1,\dots,y_n))
\end{equation}
in the $(q_{j_1},\dots,q_{j_n},\bar q_{j_1},\dots,\bar q_{j_n})$-space by letting
\begin{equation}
q_{\textbf{j}_i}=\sqrt{y_i}e^{\sqrt{-1}x_i},\qquad 1\leq i\leq n,\;\;\textbf{j}_i \in S,
\end{equation}
and
\begin{equation}
q_\textbf{j}=z_\textbf{j},\quad \bar q_\textbf{j}=\bar z_\textbf{j},\qquad \textbf{j} \in Z_1^d.
\end{equation} So system (\ref{031401}) becomes
\begin{eqnarray}\label{031402}
\frac{dy_i}{dt}&=&-P_{\theta_i},\\
\frac{dx_i}{dt}&=&\omega_i+P_{I_i},\qquad i=1,\dots,n,\\
\frac{dz_\textbf{j}}{dt}&=&-\sqrt{-1}(\Omega_\textbf{j}z_\textbf{j}+\varepsilon P_{\bar z_\textbf{j}}),\\
 \frac{d\bar z_\textbf{j}}{dt}&=&\sqrt{-1}(\Omega_\textbf{j}\bar z_\textbf{j}+\varepsilon P_{\bar z_\textbf{j}}), \qquad \textbf{j}\in \mathbb{Z}^d_1,
\label{031403}
\end{eqnarray}
where $P(x,y,z,\bar z)$ is just $G(q,\bar q)$ with the $(q,\bar q)$-variables expressed in terms of the $(x,y,z,\bar z)$-variables. The Hamiltonian associated to (\ref{031402})-(\ref{031403}) (with respect to the symplectic structure $\sum_{i=1}^ndy_i\wedge dx_i+\sqrt{-1} \sum_{\textbf{j}\in\mathbb{Z}^d_1}dz_\textbf{j} \wedge d\bar z_\textbf{j}$) is given by
\begin{equation}
H(x,y,z,\bar z;\xi)=\langle \omega(\xi),y\rangle+\sum_{\textbf{j}\in\mathbb{Z}^d_1}\Omega_{\textbf{j}}(\xi)z_\textbf{j}\bar z_\textbf{j}+P(x,y,z,\bar z;\xi).
\end{equation}
Based on (3.12) in \cite{GY2006}, the relationship (\ref{04051}) is satisfied. Note that $G(q, \bar{q})$ has $p$-tame property, and introducing action-angle variables is a coordinate symplectic transformation, so $P(x,y,z,\bar{z})$ has $p$-tame property.

Finally, we obtain a Hamiltonian $H(x,y,z,\bar z;\xi)$ having the
following form
\begin{equation}H(x,y,z,\bar{z};\xi)=N(x,y,z,\bar z;\xi)+P(x,y,z,\bar
z;\xi),\end{equation} where
\begin{equation}
N(x,y,z,\bar z;\xi)=H_{0}(w,\bar w)=\sum_{i=1}^n
\omega_i(\xi)y_i+\sum_{\textbf{j}\in\mathbb{Z}^d_1}\Omega_\textbf{j}(\xi)z_\textbf{j}\bar z_\textbf{j},
\end{equation}with the tangent frequency \begin{equation}\label{001}
\omega({\xi})=(\omega_i(\xi))_{1\leq i\leq n}, \qquad \omega_i=|\textbf{j}_i|_2^2+\xi_{\textbf{j}_i},
\end{equation}and the normal frequency
\begin{equation}\label{002}
\Omega(\xi)=(\Omega_\textbf{j}(\xi))_{\textbf{j}\in\mathbb{Z}^d_1},\qquad
\Omega_\textbf{j}(\xi)=|\textbf{j}|_2^2+\xi_{\textbf{j}},
\end{equation}

In view of (\ref{001}) and (\ref{002}), Assumption (1) in Theorem \ref{T1} satisfies.

In view of (\ref{00001}), (\ref{00002}) and noting that the coordinate transformation of action-angle variables preserves $p$-tame property, $R=\varepsilon G$ satisfy Assumption (2) in Theorem \ref{T1}.

Moreover, based on (3.12) in \cite{GY2006}, Assumption (3) in Theorem \ref{T1} satisfies.

Hence, all assumptions in Theorem \ref{T1} hold. According to Theorem \ref{T1}, we obtain a KAM normal form of order 2, where the nonlinear terms satisfy $p$-tame property.

Furthermore, we obtain a KAM partial normal form of order $\mathcal{M}+2$ where the nonlinear terms satisfy $p$-tame property based on Theorem \ref{thm7.1}.

Finally, based on Theorem \ref{T3}, for each $\xi\in\tilde\Pi\subset\Pi_{\eta}$, the KAM
torus $\mathcal{T}_{\xi}$ for equation (\ref{26}) is sticky, i.e.
for any solution
$u(t,x)$ of equation (\ref{26}) with the
initial datum satisfying
$${d}_{H^p_0[0,\pi]}(u(0,x),\mathcal{T}_{\xi})\leq \delta,$$
then
\begin{equation*}{d}_{H^p_0[0,\pi]}(u(t,x),\mathcal{T}_{\xi})\leq
2\delta,\qquad \mbox{for all} \ |t|\leq
{\delta}^{-\mathcal{M}}.
\end{equation*}
\end{proof}

\section{The measure of the non-resonant set $\tilde\Pi$}\label{091406}
In this section, we will show that for most $\xi$, the frequencies $\breve\omega(\xi)$
and $\breve\Omega(\xi)$ are $(\tilde\eta,\mathcal{N},\mathcal{M})$-non-resonant. More precisely, we have the following lemma:
\begin{lemma}
The non-resonant set $\tilde\Pi$ defined in (\ref{032501}) satisfies the following estimate
\begin{equation}\label{091301} Meas\ \tilde\Pi \geq(Meas\ \Pi_{\eta})(1-c\tilde\eta),
\end{equation}
where $c>0$ is a constant depending on $n$.
\end{lemma}
\begin{proof}
Firstly, we will show the frequencies $\breve\omega(\xi)$
and $\breve\Omega(\xi)$ are twist about the parameter $\xi$. Precisely,
in view of (\ref{012801}) and (\ref{080103}) we have
\begin{equation}\label{00003}
|\partial_{\xi_{\textbf{j}_i}}\breve\omega_{i}(\xi)|\geq 1-c\eta^8\epsilon,\qquad 1\leq i\leq n,
\end{equation}
and
\begin{equation}\label{00004}
|\partial_{\xi_{\textbf{j}}}\breve\omega_{i}(\xi)|\leq c\eta^8\epsilon,\qquad \textbf{j}\neq \textbf{j}_i.
\end{equation}
Moreover, in view of (\ref{012802}) and (\ref{080104}), we have
\begin{equation}\label{00005}
|\partial_{\xi_{\textbf{j}}}\breve\Omega_{\textbf{j}}(\xi)|\geq 1-\frac{c\eta^8\epsilon}{|\textbf{j}|_2^2},\qquad \textbf{j}\in\mathbb{Z}^d_1,
\end{equation}
and
\begin{equation}\label{00006}
|\partial_{\xi_{\textbf{j}'}}\breve\Omega_{\textbf{j}}(\xi)|\leq \frac{c\eta^8\epsilon}{|\textbf{j}'|_2^2},\qquad \textbf{j}'\neq \textbf{j}, \ \textbf{j}'\in\mathbb{Z}^d,\ \textbf{j}\in\mathbb{Z}_1^d.
\end{equation}

Secondly, we will estimate the measure of the resonant sets
$\mathcal{R}_{k\tilde l\hat l}$.
\\$\textbf{Case 1.}$\\
For $|k|\neq0$, without loss of generality, we assume
\begin{equation}\label{091206}
|k_1|=\max_{1\leq i\leq n}\{|k_1|,\dots,|k_n|\}. \end{equation} Then
\begin{eqnarray*}
&&|\partial_{\xi_{\textbf{j}_1}}(\langle k,\breve\omega(\xi)\rangle+\langle
\tilde
l,\tilde{\Omega}(\xi)\rangle+\langle\hat l,\hat{\Omega}(\xi)\rangle)|\\
&\geq&
|k_1||\partial_{\xi_{\textbf{j}_1}}\breve\omega_1(\xi)|
-\left|\partial_{\xi_{\textbf{j}_1}}\left(\sum_{i=2}^nk_i\breve\omega_i(\xi)+\langle
\tilde l,\tilde{\Omega}(\xi)\rangle+\langle\hat l,\hat{\Omega}(\xi)\rangle\right)\right|\\
&\geq&|k_1|(1-c\eta^8\epsilon)-\left(\sum_{i=2}^n|k_i|+|\tilde l|+|\hat l|\right)c\eta^{8}\epsilon\qquad \\
&&\mbox{(in view of (\ref{00003})- (\ref{00006}))}\\
&\geq& |k_1|-(|k|+\mathcal{M}+2)c\eta^{8}\epsilon\qquad \mbox{(in view of $|\tilde l|+|\hat l|\leq \mathcal{M}+2$)}\\&\geq&
\frac14|k_1|\qquad\qquad\qquad \mbox{(by
(\ref{091206}) and
$\mathcal{M}\leq(2c\eta^{8}\epsilon)^{-1})$}\\
&\geq&\frac14.
\end{eqnarray*}
Hence,
\begin{equation}\label{101102}
Meas\ \mathcal{R}_{k\tilde l\hat
l}\leq\frac{4\tilde\eta}{4^{3\mathcal{M}}(|
k|+1)^{\tau}C(\mathcal{N},\tilde l)}\cdot Meas\ \Pi_{\eta}.
\end{equation}
$\textbf{Case 2.}$ \\
If $|k|=0$ and $|\tilde l|\neq0$, without loss of generality, we
assume \begin{equation*}|\tilde {l}_{\textbf{j}'}|\neq 0
\end{equation*}
and let $$A:=\{\; \textbf{j}\;\; | \;1 \leq |\textbf{j}|_2 \leq \mathcal{N}, \textbf{j} \in Z_1^d\},\;\; A_1:= A\setminus \{\textbf{j}'\}.$$

Then
\begin{eqnarray*}
&&|\partial_{\xi_{\textbf{j}'}}(\langle k,\breve\omega(\xi)\rangle+\langle
\tilde
l,\tilde{\Omega}(\xi)\rangle+\langle\hat l,\hat{\Omega}(\xi)\rangle)|\\
&\geq& |\tilde
l_{\textbf{j}'}||\partial_{\xi_{\textbf{j}'}}\breve{\Omega}_{\textbf{j}'}(\xi)|-|\partial_{\xi_{\textbf{j}'}}(\langle
\tilde l,\tilde{\Omega}(\xi)\rangle+\langle\hat l,\hat{\Omega}(\xi)\rangle-\tilde l_{\textbf{j}'}\breve{\Omega}_{\textbf{j}'}(\xi))|\\
&\geq&|\tilde l_{\textbf{j}'}| \left(1 - \frac{c\eta^8\epsilon}{|\textbf{j}'|^2_2} \right)-\left(\sum_{\textbf{i}\in A_1}|\tilde l_\textbf{i}|+|\hat l|\right)\frac{c\eta^8\epsilon}{|\textbf{j}'|^2_2} \qquad\\
&& \mbox{(by (\ref{00005}) and (\ref{00006}))}\\
&\geq& |\tilde l_{\textbf{j}'}|-\left(|\tilde{l}|+|\hat{l}|\right)\frac{c\eta^8\epsilon}{|\textbf{j}'|^2_2}\qquad\\
&\geq& |\tilde l_{\textbf{j}'}|-\left(\mathcal{M}+2\right)\frac{c\eta^8\epsilon}{|\textbf{j}'|^2_2}\qquad\mbox{(in view of $|\tilde l|+|\hat l|\leq\mathcal{M}+2$)}\\
&\geq&|\tilde l_{\textbf{j}'}|-\frac{3}{4|\textbf{j}'|_2^2}\qquad\qquad\qquad \mbox{(in
view of $\mathcal{M}\leq(2c\eta^{8}\epsilon)^{-1}$ and $|\textbf{j}'|_2\leq \mathcal{N}$)}\\
&\geq&\frac1{4}.
\end{eqnarray*}
Hence,
\begin{equation}\label{101103}
Meas\ \mathcal{R}_{0\tilde l\hat
l}\leq\frac{4\tilde\eta}{4^{3\mathcal{M}}C(\mathcal{N},\tilde l)}\cdot Meas \
\Pi_{\eta}.
\end{equation}
$\textbf{Case 3.}$ \\If $|k|=0,|\tilde l|=0$ and
$1\leq|\hat{l}|\leq2$, then it is easy to see that $ |\langle \hat
l, \hat \Omega(\xi)\rangle|$ is not small, i.e. \begin{equation}\label{101104}\mbox{the sets
$\mathcal{R}_{k\tilde l\hat l}$ are empty for $|k|=0,|\tilde l|=0$
and $1\leq|\hat{l}|\leq2$.}\end{equation}

Now we would like to estimate the measure of $\mathcal{R}$ (see (\ref{0912081})).
Following the notations in \cite{Bam2011}, we define the set
\begin{equation*}
\mathcal{Z}_{n,\mathcal{N}}:=\left\{(k,\tilde l,\hat l)\in\mathbb{Z}^n\times\mathbb{Z}^{\mathcal{N}}\times\mathbb{Z}^{\mathbb{N}}\setminus(0,0,0):|\hat l|\leq 2\right\}
\end{equation*}
and we split
\begin{equation*}
\mathcal{L}:=\left\{\hat l\in\mathbb{Z}^{\mathbb{N}}:|\hat l|\leq 2\right\}
\end{equation*}
as the union of the following four disjoint sets:
\begin{eqnarray*}
\mathcal{L}_0&=&\{\hat l=0\},\\
\mathcal{L}_1&=&\{\hat l=e_\textbf{j}\},\\
\mathcal{L}_{2+}&=&\{\hat l=e_\textbf{i}+e_\textbf{j}\},\\
\mathcal{L}_{2-}&=&\{\hat l=e_\textbf{i}-e_\textbf{j},\textbf{i}\neq \textbf{j}\},
\end{eqnarray*}
where
\begin{equation*}
e_\textbf{j}: \text{the}\;\; \textbf{j}-th \;\;\text{position is}\;\; 1,
\end{equation*}
and $|\textbf{i}|_2,|\textbf{j}|_2\geq n+\mathcal{N}+1$.

Let $|\hat l|=2$ and $\hat{l}=e_i+e_j\in{\mathcal{L}}_{2+}$ for some $|\textbf{i}|_2,|\textbf{j}|_2\geq n+\mathcal{N}+1$.
If $$\min\{|\textbf{i}|_2^2,|\textbf{j}|_2^2\}\geq |k|\cdot||\breve{\omega}({\xi})||+2(\mathcal{M}+2)\mathcal{N}^2+1,$$
 then it is easy to see that
\begin{equation*}
\left|\langle
k,\breve\omega(\xi)\rangle+\langle \tilde
l,\tilde{\Omega}(\xi)\rangle+\langle
\hat{l},\hat{\Omega}(\xi)\rangle\right|\geq 1,
\end{equation*}which is not small. Namely, the resonant sets $\mathcal{R}_{k\tilde l\hat l}$ is empty. So it is sufficient to consider $$\max\{|\textbf{i}|_2^2,|\textbf{j}|_2^2\}< |k|\cdot\parallel\breve{\omega}({\xi})\parallel+2(\mathcal{M}+2)\mathcal{N}^2+1,$$
when the estimate (\ref{022418}) is given below. In fact, we obtain
\begin{eqnarray}
&&\nonumber Meas \bigcup_{(k,\tilde l,\hat l)\in \mathcal{Z}_{n,\mathcal{N}}\bigcap{\mathcal{L}}_{2+}}\mathcal{R}_{k\tilde l\hat l}\\&\leq& \nonumber\sum_{k\neq 0, (k,\tilde l,\hat l)\in \mathcal{Z}_{n,\mathcal{N}}\bigcap{\mathcal{L}}_{2+}} \frac{4\tilde\eta}{4^{3\mathcal{M}}(|
k|+1)^{\tau}C(\mathcal{N},\tilde l)}\cdot Meas\ \Pi_{\eta}\\
&&\nonumber+\sum_{k=0,(k,\tilde l,\hat l)\in \mathcal{Z}_{n,\mathcal{N}}\bigcap{\mathcal{L}}_{2+}} \frac{4\tilde\eta}{4^{3\mathcal{M}}C(\mathcal{N},\tilde l)}\cdot Meas \
\Pi_{\eta}\\&\leq &c_1\tilde \eta \cdot Meas \
\Pi_{\eta},\label{022418}
\end{eqnarray}where $c_1>0$ is a constant depending on $n$ and $\tau.$

Similarly we obtain
\begin{equation}
Meas \bigcup_{(k,\tilde l,\hat l)\in \mathcal{Z}_{n,\mathcal{N}}\bigcap{\mathcal{L}}_{0}}\mathcal{R}_{k\tilde l\hat l}\leq c_2\tilde\eta \cdot Meas \
\Pi_{\eta},
\end{equation}
and
\begin{equation}
Meas \bigcup_{(k,\tilde l,\hat l)\in \mathcal{Z}_{n,\mathcal{N}}\bigcap{\mathcal{L}}_{1}}\mathcal{R}_{k\tilde l\hat l}\leq c_2\tilde\eta \cdot Meas \
\Pi_{\eta},
\end{equation}
where $c_2>0$ is a constant depending on $n$ and $\tau$.
Now let
\begin{equation*}
(k,\tilde l,\hat l)\in \mathcal{Z}_{n,\mathcal{N}}\bigcap{\mathcal{L}}_{2-},
\end{equation*}
and assume $|\textbf{i}|_2>|\textbf{j}|_2$ without loss generality.
In view of (\ref{012802}) and (\ref{080104}),
there is a constant $C>0$ such that
\begin{equation*}
\left|\frac{\breve \Omega_\textbf{i}(\xi)-\breve \Omega_\textbf{j}(\xi)}{|\textbf{i}|_2^2-|\textbf{j}|_2^2}-1\right|\leq \frac C{|\textbf{j}|_2^2}.
\end{equation*}
Hence,
\begin{equation*}
\langle\hat l,\hat\Omega(\xi)\rangle=\breve\Omega_\textbf{i}(\xi)-\breve\Omega_\textbf{j}(\xi)=|\textbf{i}|_2^2-|\textbf{j}|_2^2+r_{\textbf{i}\textbf{j}},
\end{equation*}
with
\begin{equation*}
|r_{\textbf{i}\textbf{j}}|\leq \frac {Cm}{|\textbf{j}|_2^2},
\end{equation*}
and $m=|\textbf{i}|_2^2-|\textbf{j}|_2^2$. Then we have
\begin{equation*}
\left|\langle k,\breve\omega(\xi)\rangle+\langle \tilde l,\tilde\Omega(\xi)\rangle+\langle\hat l,\hat\Omega(\xi)\rangle\right|\geq \left|\langle k,\breve\omega(\xi)\rangle+\langle \tilde l,\tilde\Omega(\xi)\rangle+m\right|-|r_{\textbf{i}\textbf{j}}|.
\end{equation*}
Therefore,
\begin{equation*}
\mathcal{R}_{k\tilde l\hat l}\subset \mathcal{Q}_{k\tilde l m\textbf{j}}:=\left\{\left|\langle k,\breve\omega(\xi)\rangle+\langle \tilde l,\tilde\Omega(\xi)\rangle+m\right|\leq \frac{\tilde\eta}{4^{3\mathcal{M}}(|
k|+1)^{\tau}C(\mathcal{N},\tilde l)}+\frac{Cm}{|\textbf{j}|_2^2}\right\}.
\end{equation*}
For $|\textbf{j}|_2\geq|\textbf{j}_0|_2$, we have
\begin{equation*}
\mathcal{Q}_{k\tilde l m\textbf{j}}\subset \mathcal{Q}_{k\tilde l m\textbf{j}_0}.
\end{equation*}
Then it is sufficient to consider
\begin{equation*}
m\leq  |k|\cdot\parallel\breve{\omega}({\xi})\parallel+2(\mathcal{M}+2)\mathcal{N}^2+1,
\end{equation*}
and let
\begin{equation*}
|\textbf{j}_0|_2=\tilde \eta^{-1/2}4^{\mathcal{M}}(|k|+1)^{\tau/2}C(\mathcal{N},\tilde l)^{1/2}.
\end{equation*}
Then following the proof of Lemma 5 in \cite{Bam2011}, we obtain
\begin{equation}\label{022419}
Meas\ \bigcup_{(k,\tilde l,\hat l)\in \mathcal{Z}_{n,\mathcal{N}}\bigcap{\mathcal{L}}_{2-}}\mathcal{R}_{k\tilde l\hat l}\leq c_3\tilde\eta^{1/2} \cdot Meas\ \Pi_{\eta},
\end{equation}
where $c_3>0$ is a constant depending on $n$ and $\tau$.
Finally, in view of (\ref{022418})-(\ref{022419}) and (\ref{0912081}),
we obtain
\begin{equation}\label{022420}
Meas\ \mathcal{R}
\leq c\tilde\eta^{1/2} \cdot Meas\ \Pi_{\eta},
\end{equation}
where $c$ is a constant depending on $c_1$, $c_2$, $c_3$, $n$ and $\tau$. Then combining (\ref{032501}) with (\ref{022420}), we finish the proof of (\ref{091301}).

\end{proof}

\section{Appendix: Properties of the Hamiltonian with $p$-tame property}
In this section, we will discuss some properties of $p$-tame norm, which are proven in \cite{CLY} (or can be proven by a parallel way).
\begin{lemma}\label{021102}({Estimation of the Poisson brackets})
Suppose that both Hamiltonian functions
$$U(x,y,z;\xi)=\sum_{\beta\in\mathbb{N}^{{\mathbb{Z}_1^d}}}U^{\beta}(x,y;\xi)z^{\beta},$$
and
$$V(x,y,z;\xi)=\sum_{\beta\in\mathbb{N}^{{\mathbb{Z}_1^d}}}V^{\beta}(x,y;\xi)z^{\beta},$$
satisfy $p$-tame property on the domain $D(s,r,r)\times\Pi$,
where
\begin{equation*}
U^{\beta}(x,y;\xi)=\sum_{\alpha\in\mathbb{N}^{{n}}}U^{\alpha\beta}(x;\xi)y^{\alpha}
,\end{equation*} and
\begin{equation*}
V^{\beta}(x,y;\xi)=\sum_{\alpha\in\mathbb{N}^{{n}}}V^{\alpha\beta}(x;\xi)y^{\alpha}
.\end{equation*} Then the Poisson bracket $\{U,V\}(x,y,z;\xi)$ of
$U(x,y,z;\xi)$ and $V(x,y,z;\xi)$ with respect to the symplectic structure $\sum_{i=1}^{n}dy_i\wedge dx_i+\sqrt{-1}\sum_{\textbf{j}\in\mathbb{Z}_1^{d}}dz_{\textbf{j}}\wedge \bar{z}_{\textbf{j}}$ has $p$-tame property on the
domain $D(s-\sigma,r-\sigma',r-\sigma')\times\Pi$ for
$0<\sigma<s,0<\sigma'<r/2$. Moreover, the following inequality holds
\begin{eqnarray}&&|||X_{\{U,V\}}|||_{p,D(s-\sigma,r-\sigma',r-\sigma')\times\Pi}^T
\nonumber\\
\label{091120}&\leq&C\max\left\{\frac 1{\sigma},\frac
{r}{\sigma'}\right\}|||X_{U}|||_{p,D(s,r,r)\times\Pi}^T
|||X_{V}|||_{p,D(s,r,r)\times\Pi}^T,
\end{eqnarray}
where $C>0$ is a constant depending on $n$.
\end{lemma}
Denote $X_U^t $ by the flow of the Hamiltonian vector field of $U(x,y,z;\xi)$. It follows from Taylor's formula that
\begin{equation}\label{081108}
V\circ X_U^t(x,y,z;\xi)=\sum_{i\geq 0}
\frac{t^{i}}{i!}V^{(i)}(x,y,z;\xi),
\end{equation}
where
\begin{equation*}
V^{(0)}(x,y,z;\xi):=V(x,y,z;\xi),\qquad V^{(i)}(x,y,z;\xi):=\{V^{(i-1)},U\}(x,y,z;\xi).
\end{equation*}
Then based on (\ref{091120}) in Theorem {\ref{021102}} and (\ref{081108}), we have the following theorem, which can be parallel proved following the proof of Theorem 3.3 in \cite{CLY}:
\begin{lemma}\label{081914}({Estimation of the symplectic transformation})
Consider two Hamiltonians $U(x,y,z;\xi)$ and $V(x,y,z;\xi)$
satisfying $p$-tame property on the domain $D(s,r,r)\times\Pi$ for
some $0<s,r\leq1$. Given $0<\sigma<s, 0<\sigma'<r/2$, suppose
\begin{equation*}\label{081112}|||X_U|||_{p,D(s,r,r)\times\Pi}^T\leq \frac1{2B},
\end{equation*}
where
\begin{equation*}\label{090504}B=4Ce\max\left\{\frac{1}{\sigma},\frac{r}{\sigma'}\right\},
\end{equation*}
and $C>0$ is the constant given in (\ref{091120}) in Theorem
\ref{021102}. Then for each $|t|\leq 1$, we have
\begin{equation*}
|||X_{V\circ
X_U^t}|||_{p,D(s-\sigma,r-\sigma',r-\sigma')\times\Pi}^T\leq
2|||X_V|||_{p,D(s,r,r)\times\Pi}^T.
\end{equation*}
\end{lemma}
 The following theorem will be used to estimate the $p$-tame norm of the solution of homological equation during KAM iterative procedure and normal form iterative procedure, which can be parallel proved following the proof of Theorem 3.4 in \cite{CLY}:
\begin{lemma}\label{0005}({The $p$-tame property of homological equation})
Consider two Hamiltonians
$$U(x,y,z;\xi)=\sum_{\alpha\in\mathbb{N}^n,\beta\in\mathbb{N}^{{\mathbb{Z}_1^d}}}
U^{\alpha\beta}(x;\xi)y^{\alpha}z^{\beta},$$ and
$$V(x,y,z;\xi)=\sum_{\alpha\in\mathbb{N}^n,\beta\in\mathbb{N}^{{\mathbb{Z}_1^d}}}
V^{\alpha\beta}(x;\xi)y^{\alpha}z^{\beta}.$$ Suppose $V(x,y,z;\xi)$
has $p$-tame property on the domain $D(s,r,r)\times \Pi$, i.e
$$|||X_V|||_{p,D(s,r,r)\times\Pi}^T<\infty.$$ For each
$\alpha\in\mathbb{N}^n,\beta\in\mathbb{N}^{{\mathbb{Z}_1^d}},k\in
\mathbb{Z}^n, \textbf{j}\in\mathbb{Z}^d$ and some fixed constant $\tau>0$, assume the
following inequality holds
\begin{equation*}\label{0007}
|\widehat{U^{\alpha\beta}}(k;\xi)|+|\partial_{\xi_{\textbf{j}}}\widehat{U^{\alpha\beta}}(k;\xi)|\leq
(|k|+1)^{\tau}(|\widehat{V^{\alpha\beta}}(k;\xi)|+|\partial_{\xi_{\textbf{j}}}\widehat{V^{\alpha\beta}}(k;\xi)|)
,\end{equation*} where $\widehat{U^{\alpha\beta}}(k;\xi)$ and
$\widehat{V^{\alpha\beta}}(k;\xi)$ are the $k$-th Fourier
coefficients of $U^{\alpha\beta}(x;\xi)$ and
$V^{\alpha\beta}(x;\xi)$, respectively. Then, $U(x,y,z;\xi)$ has
$p$-tame property on the domain $D(s-\sigma,r,r)\times\Pi$ for
$0<\sigma<s$. Moreover, we have
\begin{equation}\label{0012} |||X_U|||_{p,D(s-\sigma,r,r)\times\Pi}^T\leq
\frac{c}{\sigma^{\tau}} |||X_V|||_{p,D(s,r,r)\times\Pi}^T,
\end{equation}
where $c>0$ is a constant depending on $s$ and $\tau$.
\end{lemma}
As in \cite{P1}, define \begin{equation}\label{032601}
\parallel w\parallel_{\mathcal{P}^p,D(s,r,r)}=\parallel x\parallel+\frac1{r^2}\parallel y\parallel+\frac1{r}\parallel z\parallel_p,
 \end{equation}for each $w=(x,y,z)\in D(s,r,r)$, and
 define the weighted norm of Hamiltonian vector field $X_U$ on the domain
$D(s,r,r)\times\Pi$ by
\begin{equation}\label{081101}
|||X_U|||_{\mathcal{P}^p,D(s,r,r)\times\Pi}=\sup_{(x,y,z;\xi)\in
D(s,r,r)\times\Pi}\parallel X_U\parallel_{\mathcal{P}^p,D(s,r,r)}.
\end{equation}
Then we have
\begin{lemma}\label{012002}({Compare $p$-tame norm with the usual weighted norm for a Hamiltonian vector field}) Give a Hamiltonian
\begin{equation*}\label{100802}U(x,y,z;\xi)=\sum_{\beta\in\mathbb{N}^{{\mathbb{Z}_1^d}}}
U^{\beta}(x,y;\xi)z^{\beta} \end{equation*} satisfying $p$-tame
property on the domain $D(s,r,r)\times\Pi$ for some $0<s,r\leq 1$.
Then we have
\begin{equation}\label{081105}
|||X_U|||_{\mathcal{P}^p,D(s,r,r)\times\Pi}\leq|||X_U|||_{\mathcal{P}^{p+2},D(s,r,r)\times\Pi}\leq
|||X_U|||_{p,D(s,r,r)\times\Pi}^T.
\end{equation}
\end{lemma}
This theorem can be parallel proved following the proof of the theorem  3.5 in \cite{CLY}. Based on Lemma A.4. In \cite{P1} and Theorem (\ref{012002}), we have the following estimate:
\begin{lemma}\label{081503}
Suppose the Hamiltonian
$$U(x,y,z;\xi)=\sum_{\beta\in\mathbb{N}^{{\mathbb{Z}_1^d}}}U^{\beta}(x,y;\xi)z^{\beta}$$
has $p$-tame property on the domain $D(s,r,r)\times\Pi$ for some
$0<s,r\leq 1$. Let $X_U^{t}$ be the phase flow generalized by the
Hamiltonian vector field $X_U$. Given $0<\sigma<s$ and
$0<\sigma'<r/2$, assume
\begin{equation*}\label{091810}|||X_U|||_{p,D(s,r,r)\times\Pi}^T<\min\{\sigma,\sigma'\}.
\end{equation*}
Then, for each $\xi\in\Pi$ and each $|t|\leq 1$, one has
\begin{equation}\label{091101}
\parallel X_U^t-id\parallel_{p,D(s-\sigma,r-\sigma',r-\sigma')}\leq
|||X_U|||_{p,D(s,r,r)\times\Pi}^T.
\end{equation}
\end{lemma}
\section*{Acknowledgements}{\it In the Autumn of 2007,
Professor H. Eliasson gave a series of lectures on KAM theory for
Hamiltonian PDEs in Fudan University. In those lectures, he proposed
to study the normal form in the neighbourhood of the invariant tori
and the nonlinear stability of the invariant tori.  The
authors are heartily grateful to Professor Eliasson.

The authors are also heartily grateful to Professor Bambusi and Professor Yuan for valuable discussions and suggestions. }

\bibliographystyle{amsplain}

\end{document}